\documentclass[12pt, leqno]{amsart}
\usepackage{amsmath}
\usepackage{amssymb}
\usepackage{amsfonts}
\usepackage{graphicx}
\usepackage{amsthm}
\usepackage{enumerate}
\usepackage{lscape}
\usepackage{dsfont}
\usepackage{color}
\usepackage{mathtools}

\newtheorem{theor}{Theorem}

\newtheorem{lemma}[theor]{Lemma}

\newtheorem{cor}[theor]{Corollary}

\DeclareMathOperator{\K}{K}

\title{Strongly not relatives K\"ahler manifolds}
\author[M. Zedda]{Michela Zedda}
\date{\today}
\address{Dipartimento di Matematica e Fisica\\ Universit\`a del Salento} 
 \email{michela.zedda@gmail.com}
\thanks{This work has been partially supported by the group G.N.S.A.G.A. of I.N.d.A.M} 

\begin{document}
\maketitle
\begin{abstract}
In this paper we study K\"ahler manifolds that are strongly not relative to any projective K\"ahler manifold, i.e. those K\"ahler manifolds that do not share a K\"ahler submanifold with any projective K\"ahler manifold even when their metric is rescaled by the multiplication by a positive constant. We prove two results which highlight some relations between this property and the existence of a full K\"ahler immersion into the infinite dimensional complex projective space.
As application we get that the $1$-parameter families of Bergman--Hartogs and Fock--Bargmann--Hartogs domains are strongly not relative to projective K\"ahler manifolds.
\end{abstract}
\section{Introduction and statement of the main results}

According to \cite{relatives}, two K\"ahler manifolds are called {\em relatives} when they share a common K\"ahler submanifold, i.e. if a complex submanifold of one of them with the induced metric is biholomorphically isometric to a complex submanifold of the other one with the induced metric. 
In his seminal paper \cite{calabi}, Calabi determined a criterion which characterizes K\"ahler manifolds admitting a K\"ahler immersion into finite or infinite dimensional complex space forms. The main tool he introduced is the {\em diastasis function} associated to a real analytic K\"ahler manifold, namely a particular K\"ahler potential characterized by being invariant under pull--back through a holomorphic map. 
Thanks to this property, the diastasis plays a key role in studying when two K\"ahler manifolds are relatives.
In \cite{umehara} Umehara proved that two finite dimensional complex space forms with holomorphic sectional curvatures of different signs can not be relatives. Although, as firstly pointed out by Bochner in \cite{bochner}, when the ambient space is allowed to be infinite dimensional, the situation is different: any K\"ahler submanifold of the infinite dimensional flat space $\ell^2(\mathds{C})$ admits a K\"ahler immersion into the infinite dimensional complex projective space.
Umehara's work has been generalized in the recent paper by X. Cheng and A. J. Di Scala \cite{fubinirelatives}, where the authors state necessary and sufficient conditions for 
complex space forms of finite dimension and different curvatures to not be relative to each others.
In \cite{relatives} A. J. Di Scala and A. Loi prove that a Hermitian symmetric space of noncompact type endowed with its Bergman metric is not relative to a projective K\"ahler manifold, i.e. a K\"ahler manifold which admit a local holomorphic and isometric (from now on {\em K\"ahler}) immersion into the {\em finite} dimensional complex projective space (see also \cite{huang} for the case of Hermitian symmetric spaces of noncompact type and Euclidean spaces), and their result has been generalized in \cite{mossa} to homogeneous bounded domains of $\mathds{C}^n$.
Throughout the paper, we say that a K\"ahler manifold is {\em projectively induced} when it admits a K\"ahler immersion into $\mathds{C}{\rm P}^{N\leq\infty}$. When we also specify that it is {\em infinite projectively induced}, we mean that the K\"ahler immersion is full into $\mathds{C}{\rm P}^{\infty}$.

In this paper we are interested in studying when a K\"ahler manifold $(M,g)$ is {\em strongly not relative} to any projective K\"ahler manifold, that is, when $(M,c\, g)$ is not relative to any projective K\"ahler manifold for any value of the constant $c>0$ multiplying the metric.

 Our first result can be viewed as a generalization of the results in \cite{relatives, mossa} and can be stated as follows:
\begin{theor}\label{hbd}
Let $(M,g)$ be a K\"ahler manifold such that $(M,\beta g)$ is infinite projectively induced for any $\beta>\beta_0\geq 0$. If $(M,g)$ and $\mathds{C}{\rm P}^n$ are not relatives for any $n<\infty$, then $(M,g)$ is strongly not relative to any projective K\"ahler manifold.
\end{theor}
Observe that in general there are not reasons for a K\"ahler manifold which is not relative to another K\"ahler manifold to remain so when its metric is rescaled. For example, consider that the complex projective space $(\mathds{C}{\rm P}^2,c\,g_{FS})$ where $g_{FS}$ is the Fubini--Study metric, for $c=\frac23$ is not relative to $(\mathds{C}{\rm P}^2,g_{FS})$, while for positive integer values of $c$ it is (see \cite{fubinirelatives} for a proof).

In order to state our second result, consider a $d$-dimensional K\"ahler manifold $(M,g)$ which admits global coordinates $\{z_1,\dots, z_d\}$ and denote by $M_{j}$ the $1$-dimensional submanifold of $M$ defined by:
$$
M_j=\{ z\in M|\, z_1=\dots=z_{j-1}=z_{j+1}=\dots=z_d=0\}.
$$
When exists, a K\"ahler immersion $f\!:M\rightarrow \mathds{C}{\rm P}^\infty$ is said to be {\em transversally full} when for any $j=1,\dots, d$, the immersion restricted to $M_{j}$ is full into $\mathds{C}{\rm P}^\infty$.

\begin{theor}\label{trfull}
Let $(M, g)$ be a K\"ahler manifold infinite projectively induced through a transversally full map. If for any $\alpha\geq \alpha_0>0$, $(M,\alpha\, g)$ is infinite projectively induced then $(M,g)$ is strongly not relative to any projective K\"ahler manifold.
\end{theor}

As a consequence of Theorem \ref{hbd} and Theorem \ref{trfull} we get that the $1$-parameter families of Bergman--Hartogs and Fock--Bargmann--Hartogs domains, which we describe in Section \ref{bbh}, are strongly not relative to any projective K\"ahler manifold (see corollaries \ref{chrel} and \ref{fbhrel} below).\\

The paper consts of three more sections. In the first one we briefly recall the definition of diastasis function and its properties we need and in the second one we prove Theorem \ref{hbd} and Theorem \ref{trfull}. Finally, in the third and last section we apply our results to Bergman--Hartogs and Fock--Bargmann--Hartogs domains.

The author is very grateful to Prof. Andrea Loi for all the interesting discussions and comments that helped her to improve the contents and the exposition.
\section{Calabi's diastasis function}
Consider a real analytic K\"ahler manifold $(M,g)$ and let $\varphi\!: U\rightarrow \mathds{R}$ be a K\"ahler potential for $g$ defined on a coordinate neighborhood $U$ around a point $p\in M$. Consider the analytic extension $\tilde\varphi\!: W\rightarrow \mathds{R}$, $\tilde\varphi(z,\bar z)=\varphi(z)$, of $\varphi$ on a neighborhood $W$ of the diagonal in $U\times \bar U$. The {\em diastasis function} ${\rm D}(z,w)$ is defined by the formula:
\begin{equation}\label{diast}
{\rm D}(z,w):=\tilde\varphi(z,\bar z)+\tilde\varphi(w,\bar w)-\tilde\varphi(z,\bar w)-\tilde\varphi(w,\bar z).
\end{equation}
Observe that since:
$$
\frac{\partial^2}{\partial z\partial \bar z}{\rm D}(z,w)=\frac{\partial^2}{\partial z\partial \bar z}\tilde \varphi\left(z,\bar z\right)=\frac{\partial^2}{\partial z\partial \bar z}\varphi\left(z\right),
$$
once one of its two entries is fixed, the diastasis is a K\"ahler potential for $g$. We denote by ${\rm D}_0(z)$ the diastasis centered at the origin. The following theorem due to Calabi \cite{calabi}, expresses the diastasis' property which is fundamental for our purpose.
\begin{theor}[E. Calabi]\label{induceddiast}
 Let $(M, g)$ and $(S,G)$ be K\"ahler manifolds and assume $G$ to be real analytic. Denote by $\omega$ and $\Omega$ the K\"ahler forms associated to $g$ and $G$ respectively. If there exists a holomorphic map $f\!:(M,g)\rightarrow (S,G)$ such that $f^*\Omega=\omega$, then the metric $g$ is real analytic. Further, denoted by ${\rm D}^M_p\!:U\rightarrow \mathds R$ and ${\rm D}^S_{f(p)}\!:V\rightarrow\mathds R$ the diastasis functions of $(M,g)$ and $(S,G)$ around $p$ and $f(p)$ respectively, we have ${\rm D}_{f(p)}^S\circ f={\rm D}^M_p$ on $f^{-1}(V)\cap U$.
\end{theor}

Consider the complex projective space $\mathds{C}{\rm P}^N_b$ of complex dimension $N\leq \infty$, with the Fubini-Study metric $g_{b}$ of holomorphic bisectional curvature $4b$ for $b>0$. When $b=1$ we denote by $g_{FS}$ and $\omega_{FS}$ the Fubini-Study metric and the Fubini-Study form respectively. Let $[Z_0,\dots,Z_N]$ be homogeneous coordinates,
$p=[1,0,\dots,0]$ and $U_0=\{Z_0\neq 0\}$. Define affine coordinates $z_1,\dots, z_N$ on $U_0$ by $z_j=Z_j/(\sqrt{b}Z_0)$. The diastasis on $U_0$ centered at the origin reads:
\begin{equation}\label{diastcp}
{\rm D}^b_0(z)=\frac{1}{b}\log\left(1+b\sum_{j=1}^N|z_j|^2\right).
\end{equation}

Due to Th. \ref{induceddiast} and the expression of $\mathds{C}{\rm P}^N_b$'s diastasis \eqref{diastcp}, if $f\!:S\rightarrow \mathds{C}{\rm P}^N_b$ is a holomorphic map, $f(z)=[f_0(z),f_1(z),\dots, f_N(z)]$, then the induced diastasis ${\rm D}^S_0$ in a neighborhood of a point $p\in S$ is given by:
$$
{\rm D}^S_0(z)=\frac{1}{b}\log\left(1+b\sum_{j=1}^N|f_j(z)|^2\right).
$$
Further, if the K\"ahler map $f$ is assumed to be {\em full}, i.e. the image $f(S)$ is not contained into any lower dimensional totally geodesic submanifold of $\mathds{C}{\rm P}_b^N$, then $f$ is univocally determined up to rigid motion of $\mathds{C}{\rm P}^N_b$ \cite[pp. 18]{calabi}:
\begin{theor}[Calabi's Rigidity]\label{local rigidityb}
If a neighborhood $V$ of a point $p$ admits a full K\"ahler immersion into $(\mathds{C}{\rm P}^N_b,g_b)$, then $N$ is univocally determined by the metric and the immersion is unique up to rigid motions of $(\mathds{C}{\rm P}^N_b,g_b)$.
\end{theor}    
Observe that by Th. \ref{local rigidityb} above, a K\"ahler manifold which is infinite projectively induced does not admit a K\"ahler immersion into any finite dimensional complex projective space.

\section{Proof of Theorems \ref{hbd} and \ref{trfull}}
\begin{proof}[Proof of Theorem \ref{hbd}]
Observe first that due to Th. \ref{induceddiast} it is enough to prove that $(M, c\,g)$ is not relative to $\mathds{C}{\rm P}^n$ for any finite $n$ and any $c>0$. For any $c>0$, we can choose a positive integer $\alpha$ such that  $c\alpha>\beta_0$. Denote by $\omega$ the K\"ahler form on $M$ associated to $g$. Let $F\!:M\rightarrow \mathds{C}{\rm P}^\infty$ be a full K\"ahler map such that $F^*\omega_{FS}=c\alpha\, \omega$. Then $\tilde F=F/\sqrt{\alpha}$ is a K\"ahler map of $(M,c\, g)$ into $ \mathds{C}{\rm P}^\infty_\alpha$.
 Let $S$ be a common K\"ahler submanifold of $(M,c\,g)$ and $\mathds{C}{\rm P}^n$. Then by Th. \ref{induceddiast} for any $p\in S$ there exist a neighborhood $U$ and two holomorhic maps $f\!:U\rightarrow M$ and $h\!:U\rightarrow \mathds{C}{\rm P}^n$, such that $f^*(c\omega)|_U=( \tilde F\circ f)^*\omega_{FS}|_U=h^*\omega_{FS}|_U$.
 
Thus, by \eqref{diastcp} one has:
$$
\log\left(1+\sum_{j=1}^n|h_j|^2\right)=\frac1\alpha\log\left(1+\sum_{j=1}^\infty|(F\circ f)_j)|^2\right).
$$
i.e.:
\begin{equation}\label{contradiction}
\alpha\log\left(1+\sum_{j=1}^n|h_j|^2\right)=\log\left(1+\sum_{j=1}^\infty|(F\circ f)_j)|^2\right).
\end{equation}
Since $F\circ f$ is full and $\alpha$ is a positive integer, this last equality and Calabi rigidity Theorem \ref{local rigidityb} imply $n=\infty$.
\end{proof}

\begin{proof}[Proof of Theorem \ref{trfull}]
Due to Th. \ref{hbd} and Th. \ref{induceddiast} we need only to prove that a if a K\"ahler manifold is infinite projectively induced through a  transversally full immersion then it is not relative to $\mathds{C}{\rm P}^n$ for any $n$. 
Assume that $S$ is a $1$-dimensional K\"ahler submanifold of both $\mathds{C}{\rm P}^n$ and $(M, g)$.
Then around each point $p\in S$ there exist an open neighborhood $U$ and  two holomorphic maps $\psi\!:U\rightarrow \mathds{C}{\rm P}^n$ and $\varphi\!:U\rightarrow M$, $\varphi(\xi)=(\varphi_1(\xi),\dots,\varphi_d(\xi))$ where $\xi$ are coordinates on $U$, such that $\psi^*\omega_{FS}|_U=\varphi^*(c\omega)|_U$. Without loss of generality we can assume $\frac{\partial\varphi_1(\xi)}{\partial \xi}(0)\neq 0$.
 Let $f\!:M\rightarrow \mathds{C}{\rm P}^\infty$ be a K\"ahler map from $(M, g)$ into $\mathds{C}{\rm P}^\infty$. Since by assumption $f$ is transversally full, $f=[f_0,\dots, f_j,\dots]$ contains for any $m=1,2,3,\dots$, a subsequence $\left\{f_{j_1},\dots, f_{j_m}\right\}$ of functions which restricted to $M_1$ are linearly independent.
The map $f\circ\varphi\!:U\rightarrow \mathds{C}{\rm P}^\infty$ is full, in fact $f|_{M_1}\circ \varphi$ is full since $\varphi_1(\xi)$ is not constant and for any $m=1,2,3,\dots$, $\left\{f_{j_1}(\varphi_1(\xi)),\dots, f_{j_m}(\varphi_1(\xi))\right\}$ is a subsequence of $\{f|_{M_1}\circ \varphi\}$ of linearly independent functions.  Conclusion follows by Calabi's rigidity Theorem \ref{local rigidityb}.
\end{proof}

\section{Applications}\label{bbh}

Let $(\Omega, \beta g_B)$, $\beta >0$, denote a bounded domain of $\mathds{C}^d$ endowed with a positive multiple of its Bergman metric $g_B$. Recall that $g_B$ is the K\"ahler metric on $\Omega$ whose associated K\"ahler form $\omega_B$ is given by $\omega_B=\frac{i}{2}\partial\bar\partial\log \K(z,z)$, where  $\K(z,z)$ is the reproducing kernel for the Hilbert space:
$$\mathcal{H} =\left\{\varphi\in{\rm hol}(\Omega),\ \int_\Omega |\varphi|^{2}\ \frac{\omega_0^d}{d!}<\infty\right\},$$
where $\omega_0=\frac{i}{2} \sum_{j=1}^d dz_j\wedge d\bar z_j$ is the standard K\"ahler form of $\mathds{C}^d$. It follows by \eqref{diast} that the diastasis function for $g_B$ is given by:
\begin{equation}\label{diastomega}
{\rm D}_0^\Omega(z)=\log\frac{\K(z,z)\K(0,0)}{|\K(z,0)|^2}.
\end{equation}
Observe that the Bergman metric $g_B$ admits a natural K\"ahler immersion into the infinite dimensional complex projective space (cfr. \cite{kodomain}). More precisely, if $\K(z,z)=\sum_{j=0}^\infty |\varphi_j(z)|^2$, the map:
\begin{equation}\label{bergmanimm}
\varphi\!:\Omega\rightarrow\mathds{C}{\rm P}^\infty,\quad \varphi=(\varphi_0,\dots, \varphi_j,\dots),
\end{equation}
is a K\"ahler immersion of $(\Omega, g_B)$ into $\mathds{C}{\rm P}^\infty$, for $\varphi^*g_{FS}=g_B$, as it follows by:
$$
\omega_B=\frac i2\partial\bar\partial\log(\K(z,z))=\frac i2\partial\bar\partial\log\left(\sum_{j=0}^\infty |\varphi_j(z)|^2\right)=\varphi^*\omega_{FS}.
$$
Further, such immersion is full since $\{\varphi_j\}$ is a basis for the Hilbert space $\mathcal H$ and a bounded  domain does not admit a K\"ahler immersion into a finite dimensional complex projective space even when the metric is rescaled. Although the existence of a K\"ahler immersion of $(\Omega, \beta g_B)$ into $\mathds{C}{\rm P}^\infty$ is strictly related to the constant $\beta$ which multiplies the metric (see \cite{articwall} for the case when $\Omega$ is symmetric). In \cite{ishi} it is proven that the only homogeneous bounded domain which is projectively induced for all positive values of the constant multiplying the metric is a product of complex hyperbolic spaces.
Although, the property of being projectively induced for a large enough constant is not so unusual and the following holds \cite{loimossaber}:
\begin{theor}[A. Loi, R. Mossa]\label{loimossaimm}
Let $(\Omega,g)$ be a homogeneous bounded domain. Then, there exists $\alpha_0>0$ such that $(\Omega,\alpha g)$ is projectively induced for any $\alpha\geq\alpha_0>0$. 
\end{theor}
Notice that it is an open question if the same statement holds dropping the homogeneous assumption. 

Regarding the property of being relative to some projective K\"ahler manifold, we recall the following result due to A. J. Di Scala and A. Loi in \cite{relatives}, which plays a key role in the proof of Corollary \ref{chrel}.
\begin{theor}[A. J. Di Scala, A. Loi ]\label{loidiscalahbd}
A bounded domain of $\mathds{C}^n$ endowed with its Bergman metric and a projective K\"ahler manifold are not relatives.
\end{theor}
Observe that due to theorems \ref{loimossaimm} and \ref{loidiscalahbd}, Theorem \ref{hbd} implies that a bounded domain of $\mathds{C}^n$ endowed with its Bergman metric and a projective K\"ahler manifold are {\em strongly} not relatives. Althought, this result has been proven in a more general context by R. Mossa in \cite{mossa}, where he shows that a homogeneous bounded domain and a projective K\"ahler manifold are not relatives.\\

Let us now describe the family of Bergman--Hartogs domains. 
For all positive real numbers $\mu$ a {\em Bergman-Hartogs domain} is defined by:
\begin{equation}\label{defm}
M_{\Omega}(\mu)=\left\{(z,w)\in \Omega\times\mathds{C},\ |w|^2<\tilde \K(z, z)^{-\mu}\right\},
\end{equation}
where $\tilde \K(z, z)=\frac{\K(z,z)\K(0,0)}{|\K(z,0)|^2}$ with $\K$ the Bergman kernel of $\Omega$.
Consider on $M_{\Omega}(\mu)$ the metric $g(\mu)$  whose associated K\"ahler form $\omega(\mu)$ can be described by the (globally defined) K\"ahler potential centered at the origin
\begin{equation}\label{diastM}
\Phi(z,w)=-\log(\tilde\K(z, z)^{-\mu}-|w|^2).
\end{equation}
The domain $\Omega$ is called  the {\em  base} of the Bergman--Hartogs domain 
$M_{\Omega}(\mu)$ (one also  says that 
$M_{\Omega}(\mu)$  is based on $\Omega$). Observe that these domains include and are a natural generalization of Cartan--Hartogs domains which have been studied under several points of view (see e.g. \cite{fengtubalanced,berezinCH} and references therein). To the author knowledge, Bergman-Hartogs domains has been already considered in \cite{hao,hao2,hao3}.

In \cite{articwall} the author of the present paper jointly with A. Loi proved that when the base domain is symmetric $(M_{\Omega}(\mu),c\,g(\mu))$ admits a K\"ahler immersion into the infinite dimensional complex projective space if and only if $(\Omega, (c+m)\mu g_B)$ does for every integer $m\geq0$. As pointed out in \cite{hao}, a totally similar proof holds also when the base is a homogeneous bounded domain. This fact together with Theorem \ref{loimossaimm} proves that a Bergman--Hartogs domain $(M_{\Omega}(\mu),c\,g(\mu))$ is projectively induced for all large enough values of the constant $c$ multiplying the metric. Further, the immersion can be written explicitely as follows (cfr. \cite[Lemma 8]{balancedch}):
\begin{lemma}\label{chimm}
Let $\alpha$ be a positive real number such that the Bergman--Hartogs domain $(M_{\Omega}(\mu),\alpha\, g(\mu))$ is projectively induced. Then, the K\"ahler map $f$ from $(M_{\Omega}(\mu),\alpha\,g(\mu))$ into $\mathds{C}{\rm P}^\infty$, up to unitary transformation of $\mathds{C}{\rm P}^\infty$, is given by:
\begin{equation}\label{immf}
f=\left[ 1, s, h_{\mu\, \alpha},\dots,\sqrt{\frac{(m+ \alpha-1)!}{(\alpha-1)!m!}}h_{\mu(\alpha +m)}w^m,\dots\right],
\end{equation}
where $s=(s_1,\dots, s_m,\dots)$ with
$$s_m=\sqrt{\frac{(m+ \alpha-1)!}{(\alpha-1)!m!}}w^m,$$
and $h_k=(h_k^1,\dots,h_k^j,\dots)$ denotes the sequence of holomorphic maps on $\Omega$ such that the immersion $\tilde h_k=(1,h_k^1,\dots, h_k^j,\dots)$, $\tilde h_k\!:\Omega\rightarrow\mathds{C}{\rm P}^\infty$, satisfies $\tilde h_k^*\omega_{FS}=k \omega_B$, i.e. 
\begin{equation}
1+\sum_{j=1}^{\infty}|h_k^j|^2=\tilde \K^{-k}.\nonumber
\end{equation} 
\end{lemma}
\begin{proof}
The proof follows essentially that of \cite[Lemma 8]{balancedch} once considered that $\Phi(z,w)=-\log(\tilde \K(z, z)^{-\mu}-|w|^2)$ is the diastasis function for $(M_\Omega(\mu), g(\mu))$ as follows readily applying \eqref{diast}.
\end{proof}

Observe that such map is full, as can be easily seen for example by considering that for any $m=1,2,3,\dots,$ the subsequence $\{s_1,\dots, s_m\}$ is composed by linearly independent functions.

As a consequence of theorems \ref{hbd}, \ref{trfull}, \ref{loidiscalahbd} and Lemma \ref{chimm}, we get the following:
\begin{cor}\label{chrel}
For any $\mu>0$, a Bergman--Hartogs domain $(M_\Omega(\mu), g(\mu))$ is strongly not relative to any projective manifold. 
\end{cor}
\begin{proof}
Observe first that due to Th. \ref{induceddiast} it is enough to prove that $(M_\Omega(\mu),\alpha g(\mu))$ is not relative to $\mathds{C}{\rm P}^n$ for any finite $n$. Further, by Th. \ref{hbd} and Th. \ref{loidiscalahbd}, a common submanifold $S$ of both $(M_\Omega(\mu),\alpha g(\mu))$ and $\mathds{C}{\rm P}^n$ is not contained into $(\Omega,\alpha g(\mu)|_\Omega)$, since $\alpha g(\mu)|_\Omega=\frac{\alpha\mu}\gamma g_B$ is a multiple of the Bergman metric on $\Omega$. Thus, due to arguments totally similar to those in the proof of Th. \ref{trfull}, it is enough to check that the K\"ahler immersion $f\!:M_\Omega(\mu)\rightarrow \mathds{C}{\rm P}^\infty$ is transversally full with respect to the $w$ coordinate. Conclusion follows then by \eqref{immf}. 
\end{proof}

Finally, we describe what we need about the $1$-parameter family of Fock--Bargmann--Hartogs domains, referring the reader to \cite{fbh} and reference therein for details and further results. For any value of $\mu>0$, a Fock--Bargmann--Hartogs domain $D_{n,m}(\mu)$ is a strongly pseudoconvex, nonhomogeneous unbounded domains in $\mathds{C}^{n+m}$ with smooth real-analytic boundary, given by:
$$
D_{n,m}(\mu):=\{(z,w)\in \mathds{C}^{n+m}: ||w||^2< e^{-\mu||z||^2}\}.
$$
One can define a K\"ahler metric $\omega(\mu;\nu)$, $\nu>-1$ on $D_{n,m}(\mu)$ through the globally defined K\"ahler potential:
$$
\Phi(z,w):=\nu\mu||z||^2-\log(e^{-\mu||z||^2}-||w||^2).
$$
In \cite{fbh}, E. Bi, Z. Feng and Z. Tu prove that when $n=1$ and $\nu=-\frac{1}{m+1}$, the metric $\omega(\mu;\nu)$  is infinite projectively induced whenever it is rescaled by a big enough constant. More precisely they prove the following:
\begin{theor}[E. Bi, Z. Feng, Z. Tu]\label{fbhth}
The metric $\alpha g(\mu;\nu)$ on the Fock--Bargmann--Hartogs domain $D_{n,m}(\mu)$ is balanced if and only if $\alpha>m+n$, $n=1$, $\nu=-\frac1{m+1}$.
\end{theor}
Recall that a balanced K\"ahler metric is a particular projectively induced metric such that the immersion map is defined by a orthonormal basis of a weighted Hilbert space (see e.g. \cite{balancedch}).  

In order to apply Th. \ref{trfull} to Fock--Bargmann--Hartogs domains we need the following lemma:
\begin{lemma}\label{focktr}
For any $\mu>0$ and any $\alpha>m+1$, a Fock--Bargmann--Hartogs domain $\left(D_{1,m}(\mu),\alpha\omega(\mu;-\frac1{m+1})\right)$ admits a transversally full K\"ahler immersion into $\mathds{C}{\rm P}^\infty$.
\end{lemma}
\begin{proof}
A K\"ahler immersion exists due to Th. \ref{fbhth}. In order to see that it is transversally full, observe that when $w_1=\dots= w_m=0$, $\alpha\omega(\mu;-\frac1{m+1})|_{M_1}$ is a multiple of the flat metric, and when only one $w_j$ is different from zero $\alpha\omega(\mu;-\frac1{m+1})|_{M_j}$ is a multiple of the hyperbolic metric.
\end{proof}
\begin{cor}\label{fbhrel}
For any $\mu>0$, a Fock--Bargmann--Hartogs domain $\left(D_{1,m}(\mu),\omega(\mu;-\frac1{m+1})\right)$ is strongly not relative to any projective manifold. 
\end{cor}
\begin{proof}
If follows directly from Th. \ref{trfull} and Lemma \ref{focktr}.
\end{proof}

\end{document}